\documentclass[12pt, a4paper]{article}

\usepackage{amsmath}
\usepackage{amsthm}
\usepackage{amssymb}

\usepackage[english]{babel}

\usepackage{tikz}

\usepackage{comment}
\usepackage{enumitem}
\usepackage{authblk}

\theoremstyle{plain}
\newtheorem{thm}{Theorem}[section]

\newtheorem{defi}[thm]{Definition}
\newtheorem{lem}[thm]{Lemma}

\theoremstyle{remark}

\newtheorem*{rem*}{Remark}

\newcommand{\N}{{\mathbb{N}}} 
\newcommand{\R}{{\mathbb{R}}} 

\newcommand{\NN}{{\mathbb{N}}} 

\DeclareSymbolFont{bbold}{U}{bbold}{m}{n}
\DeclareSymbolFontAlphabet{\mathbbold}{bbold}

\numberwithin{equation}{section}

\newcommand{\Addresses}{{
  \bigskip
  \footnotesize

	\noindent
  Sigrid Grepstad, \textsc{Department of Mathematical Sciences, Norwegian University of 
	Science and Technology, 7491 Trondheim, Norway.} 
	\par\nopagebreak
	\noindent \textit{E-mail address: }\texttt{sigrid.grepstad@ntnu.no}

  \medskip
	\noindent
  Lisa Kaltenb\"ock and Mario Neum\"uller, \textsc{Department of Financial Mathematics and applied 
	Number Theory, Johannes Kepler University, Altenbergerstra{\ss}e 69, 4040 
	Linz, Austria}
  \par\nopagebreak
  \noindent \textit{E-mail addresses: }\texttt{lisa.kaltenboeck@jku.at} and 
	\texttt{mario.neumueller@jku.at}}}

\allowdisplaybreaks

\title{A positive lower bound for $\liminf_{N\to\infty} \prod_{r=1}^N \left| 2\sin \pi r \varphi \right|$}

\author{Sigrid Grepstad, Lisa Kaltenb\"ock and Mario Neum\"uller \thanks{This work was funded 
by the Austrian Science Fund (FWF): Project F5507-N26 and Project F5509-N26, which is part of the Special Research 
Program ``Quasi-Monte Carlo Methods: Theory and Applications''.}}

\date{}

\begin{document}

	\maketitle

	\begin{abstract}
	Nearly 60 years ago, Erd\H{o}s and Szekeres raised the question of whether
	 	$$\liminf_{N\to \infty} \prod_{r=1}^N \left| 2\sin \pi r \alpha \right| =0$$
	for all irrationals $\alpha$ \cite{erdos}. Despite its simple formulation, the question has remained unanswered. It was shown by Lubinsky in 1999 that the answer is yes if $\alpha$ has unbounded continued fraction coefficients, and it was suggested that the answer is yes in general \cite{Lu99}. 
However, we show in this paper that for the golden ratio $\varphi=(\sqrt{5}-1)/2$,
		$$\liminf_{N\to \infty} \prod_{r=1}^N \left| 2\sin \pi r \varphi \right| >0 ,$$
providing a negative answer to this long-standing open problem.

	\end{abstract}

	\centerline{\begin{minipage}[hc]{130mm}{
	}
	\end{minipage}}


\section{Introduction}
In this paper we study the asymptotic behaviour of the sequence of sine products 
\begin{equation}\label{Def:PN}
	P_N(\alpha):=\prod_{r=1}^{N}|2\sin(\pi r \alpha)|,
\end{equation}
where $N\in \N$, and $\alpha \in \R$ is fixed. For rational $\alpha = p/q$ it is clear that 
$P_N(\alpha)=0$ for all $N\geq q$, so we restrict our attention to irrational $\alpha$. Moreover, since $P_N(\alpha)=P_N(\{\alpha\})$, where $\{\cdot\}$ denotes the fractional part, we consider only $0<\alpha <1$.

The study of the sequence $P_N(\alpha)$ goes back to the late 1950s, when questions about its asymptotic behaviour were raised by Erd\H{o}s and Szekeres \cite{erdos}. 
Another early exposition on $P_N(\alpha)$ was given by Sudler in \cite{Su64}, giving rise to the name \emph{Sudler product}. 
The continued analysis of $P_N(\alpha)$ has been carried out in a number of different fields in both pure and applied mathematics (such as partition theory \cite{Su64, Wr64}, Pad\'{e} approximation and continued fractions \cite{Lu99}, as well as KAM theory and the theory of strange non-chaotic attractors \cite{AwHiSh13,GrOtPeYo84,KnTa2012,KuAr95}). This broad interest in the Sudler product has lead to a range of different notations and terminologies, making it challenging to get a full picture of what is actually known. For a compact survey of central results on $P_N(\alpha)$, we recommend the introduction of \cite{MV16}. For a survey on the more general product
\begin{equation*}
P_N\left((x_k)_{k\in\NN}\right) := \prod_{r=1}^N |2\sin(\pi x_r)|,
\end{equation*}
where $(x_k)_{k\in \N}$ is a uniformly distributed sequence in the unit interval, we recommend \cite{ALPET18}.

A long-standing open question raised by Erd\H{o}s and Szekeres in 1959 is: what can we say about $\liminf_{N \to \infty} P_N(\alpha)$? This question occupied Lubinsky, who studied the product $P_N(\alpha)$ in the context of $q$-series in \cite{Lu99}. In his paper, Lubinsky shows that if $\alpha$ has \emph{unbounded} continued fraction coefficients, then surely 
\begin{equation*}
\liminf_{N \to \infty} \prod_{r=1}^N |2\sin \pi r \alpha| =0 .
\end{equation*} 
Moreover, he expresses that he ``feels certain that it is true in general''.
%

The main goal of this paper is to show that, in fact, this is not the case. 
\begin{thm}\label{thm:CountExLub}
	If $\varphi=(\sqrt{5}-1)/2$, then
	\begin{equation}
		\liminf_{N\to\infty} P_N(\varphi) = \liminf_{N\to\infty} \prod_{r=1}^N \left| 2\sin(\pi r \varphi) \right|>0.
	\end{equation}
\end{thm}
The number $\varphi=(\sqrt{5}-1)/2$, known as the fractional part of the golden ratio, has the simplest possible continued fraction expansion 
$$\varphi= \frac{1}{\displaystyle 1+\frac{1}{\displaystyle 1+\frac{1}{\displaystyle 1+ \ldots}}} 
	= [0;\overline{1}].$$
This observation is key in establishing Theorem \ref{thm:CountExLub}. Nevertheless, we suspect that $\liminf_{N\to \infty} P_N(\alpha)>0$ also for other quadratic irrationals $\alpha$ (see Section~\ref{sec:quadirrat} for a discussion on this).

In the following section, we present our strategy for proving Theorem~\ref{thm:CountExLub}. 
The proof relies heavily on a paper by Mestel and Verschueren \cite{MV16}, where the asymptotic behaviour of the subsequence $P_{F_n}(\varphi)$ is investigated for the Fibonacci sequence $(F_n)_{n\in \N_0}$.
Let us therefore briefly review the connection between the golden ratio $\varphi$ and the Fibonacci sequence before we present our proof strategy.

\subsection{The Fibonacci sequence}
Throughout this paper, we denote by $\varphi$ the (fractional part of the) golden ratio 
\begin{equation*}
\varphi := \frac{\sqrt{5}-1}{2},
\end{equation*}
and by $(F_n)_{n\in \N_0} = (0,1,1,2,3,5,8,13, \ldots)$ the sequence of Fibonacci numbers. There is an intimate relationship between $\varphi$ and the Fibonacci sequence; $(F_n)_{n\geq 1}$ is precisely the sequence of best approximation denominators of $\varphi$. Moreover, we have the property
	\begin{equation}\label{eq:GoldenRatioFibonacci}
		F_n\varphi= F_{n-1} - (-\varphi)^{n},
	\end{equation}
for all values of $n\in\N$.

Finally, recall that any positive integer $N$ has a unique expansion in terms of the Fibonacci sequence, known as its \emph{Zeckendorf representation} \cite{Z92}.
\begin{defi}
\label{def:zeckendorf}
Any $N \in \N$ has a unique Zeckendorf representation 
\begin{equation*}
N=\sum_{j=1}^m F_{n_j},
\end{equation*}
where $(F_n)_{n\in \N_0}$ is the Fibonacci sequence, and:
\begin{enumerate}[label=(\roman*)]
\item $n_1 \geq 2$;
\item \label{item:jump} $n_{j+1}>n_j+1$ for all $j \in \{1, \ldots , m-1\}$.
\end{enumerate}
Moreover, it is well known that $n_m = \mathcal{O}(\log N)$ (see e.g. \cite[p.~126]{KN74}).
\end{defi}
In other words, we can associate to any $N \in \N$ a unique integer sequence $(n_1, \ldots , n_{m})$. Note that since $m<n_m$, the length of this sequence is $m = \mathcal{O}(\log N)$.


\section{Strategy}\label{sec:strategy}
The proof of Theorem \ref{thm:CountExLub} relies on central results in a recent paper by Mestel and Verschueren \cite{MV16}. In this paper, the authors analyse the asymptotic behaviour of the product sequence $\left(P_N(\varphi)\right)_{N\geq 1}$ for the golden ratio $\varphi$, and show in particular that:
\begin{thm}[{\cite[Theorem 3.1]{MV16}}]\label{thm:MainResultMV}
	The subsequence $\left(P_{F_n}(\varphi)\right)_{n\geq 1}$ is convergent, and 
	\begin{equation}
		\lim_{n\to\infty}P_{F_n}(\varphi)=\lim_{n\to \infty}\prod_{r=1}^{F_n}|2\sin(\pi r \varphi)|
		=2.407\ldots \,. 
	\end{equation}
\end{thm}
A consequence of Theorem \ref{thm:MainResultMV} is that the general product $P_N(\varphi)$ must necessarily obey polynomial bounds\footnote{This was first established by Lubinsky in \cite{Lu99} using a different approach.} 
\begin{equation}
\label{eq:polybounds}
N^{C_1} \leq P_N(\varphi) \leq N^{C_2} ,
\end{equation}	
where $C_1 \leq 0 < 1 \leq C_2$. These bounds are established as follows: expressing the integer $N$ by its Zeckendorf representation $N= \sum_{j=1}^m F_{n_j}$, we can rewrite $P_N(\varphi)$ as 
\begin{equation}
\label{eq:prod}
P_N(\varphi) = \prod_{j=1}^{m} \prod_{r=1}^{F_{n_j}} \left| 2 \sin \pi (r\varphi + k_j \varphi)\right|,
\end{equation} 
where $k_j = \sum_{s=j+1}^m F_{n_s}$ for $1 \leq j \leq m-1$ and $k_m=0$ (see \cite[p.~220]{MV16} for further details). Mestel and Verschueren then show that:
\begin{lem}[{see \cite[p.~220-221]{MV16}}]
\label{lem:boundshift}
There exist real constants $0<K_1\leq 1 \leq K_2$ (independent of $N$) bounding all terms in \eqref{eq:prod}, i.e.\ so that
\begin{equation}
\label{eq:shiftbound}
K_1 \leq \prod_{r=1}^{F_{n_j}} |2 \sin \pi (r\varphi+ k_j \varphi)| \leq K_2,
\end{equation}
for all $1 \leq j \leq m$.
\end{lem}
It immediately follows from Lemma \ref{lem:boundshift} and \eqref{eq:prod} that
\begin{equation*}
K_1^m \leq P_N(\varphi) \leq K_2^{m}.
\end{equation*}
Finally, since the Zeckendorf representation of $N$ has length $m =\mathcal{O}(\log N)$, we get \eqref{eq:polybounds} for some constants $C_1 < C_2$. It follows immediately from Theorem \ref{thm:MainResultMV} that $C_1 \leq 0$ (and an argument of why $C_2\geq 1$ is given in \cite[p.~219]{MV16}).

Our strategy for concluding that $\liminf_{N\to \infty} P_N(\varphi) >0$ is to evaluate the subproducts in \eqref{eq:prod} more carefully for large values of $j$. 
\begin{lem}
\label{lem:boundslarge}
There exists a threshold value $J \in \N$ (independent of $N$) such that for all terms in \eqref{eq:prod} where $j> J$, we have
\begin{equation*}
\prod_{r=1}^{F_{n_j}} |2 \sin \pi (r\varphi+ k_j \varphi)| \geq 1 .
\end{equation*} 
\end{lem}
Combining Lemmas \ref{lem:boundshift} and \ref{lem:boundslarge}, we find that 
\begin{equation*}
P_N(\varphi) = \prod_{j=1}^{m} \prod_{r=1}^{F_{n_j}} \left| 2 \sin \pi (r\varphi + k_j \varphi)\right| \geq K_1^{J} > 0,
\end{equation*}
confirming Theorem \ref{thm:CountExLub}.

The proof of Lemma \ref{lem:boundslarge} is given in Section \ref{sec:ProofMainThm}. It requires a certain decomposition of the product $\prod_{r=1}^{F_{n_j}} \left| 2 \sin \pi (r\varphi + k_j \varphi)\right|$ into three more manageable subproducts. This decomposition is inspired by the work of Mestel and Verschueren, and is thoroughly described in the following section. 


\section{Decomposition}
It is shown in \cite[Lemma 5.1]{MV16} that the product $P_{F_n}(\varphi)$ can be split into three subproducts
	\begin{equation}\label{eq:DecompPFn}
		P_{F_n}(\varphi)= \prod_{r=1}^{F_n} \left| 2 \sin \pi r \varphi \right| = A_nB_nC_n,
	\end{equation}
where
	\begin{align}
		A_n &=|2F_n\sin(\pi \varphi^n)| \label{Def:An},\\
		B_n &= \prod_{t=1}^{F_n-1}\left|\frac{s_{nt}}{2\sin(\pi t/F_n)}\right| \label{Def:Bn},\\
		C_n &= \prod_{t=1}^{(F_n-1)/2}\left(1-\frac{s_{n0}^2}{s_{nt}^2}\right) \label{Def:Cn},
	\end{align}
and where
	\begin{equation}\label{def:sn}
		s_{nt}:= 2\sin\pi\left(\frac{t}{F_n}-\varphi^n\left(\left\{\frac{tF_{n-1}}{F_n}\right\}-1/2
			\right)\right).
	\end{equation}
Note that whenever $F_n$ is even, the notation $\prod_{t=1}^{(F_n-1)/2}$ in \eqref{Def:Cn} indicates that the final term is raised to the power $1/2$. Morever, we point out that $s_{nt}$ necessarily satisfies $s_{nt} = s_{n(F_n-t)}$ for $t \in \{ 1, \ldots , F_n-1\}$.

A similar decomposition can be established for a perturbed version of $P_{F_n}(\varphi)$. Let us introduce the notation 
\begin{equation}
\label{eq:shiftedpn}
P_{F_n}(\varphi, \varepsilon) = \prod_{r=1}^{F_n} \left| 2 \sin \pi (r\varphi+ \varepsilon) \right| ,
\end{equation}
where $\varepsilon$ is some fixed, real number.
We claim the following: 
\begin{lem}
We have 
	\begin{equation}\label{eq:PerturbedProductDecomposition}
		 P_{F_n}(\varphi,\varepsilon)= \overline{A}_{n}(\varepsilon)B_{n}\overline{C}_n(\varepsilon),
	\end{equation}
	where 
	\begin{align}
		\overline{A}_n(\varepsilon)=2F_{n}|\sin \pi((-\varphi)^{n}-\varepsilon)| \label{eq:DefAiEpsi},\\
		\overline{C}_n(\varepsilon)=\prod_{t=1}^{(F_{n}-1)/2}\left(1-\frac{v_{n}^2}{s_{nt}^2}\right)\label{eq:DefCiEpsi},
	\end{align}
	$B_n$ and $s_{nt}$ are given in \eqref{Def:Bn} and \eqref{def:sn} respectively, and 
	\begin{equation}
	\label{def:vn} 
		v_n:=2\sin \pi \left(\frac{(-\varphi)^{n}}{2} - \varepsilon\right).
	\end{equation}
\end{lem}
\begin{proof}
By definition we have that
\begin{align*}
	(P_{F_{n}}(\varphi,\varepsilon))^2
	=&(2 \sin \pi ( F_{n} \varphi + \varepsilon))^2\!\!
		\prod_{r=1}^{F_n-1}\!\! \left( 2 \sin \pi (r\varphi+ \varepsilon) \right)^2\\
	=&(2 \sin \pi ( F_{n} \varphi + \varepsilon))^2\!\!
		\prod_{r=1}^{F_{n}-1} \!\!\big(2 \sin\pi (r \varphi + \varepsilon)\big)\hspace{-0.05cm}
		\big(2 \sin\pi((F_{n}-r)\varphi+\varepsilon)\big)\\
	=&(2 \sin \pi ( F_{n} \varphi + \varepsilon))^2\!\!
		\prod_{r=1}^{F_{n}-1} \!\!2 \big( \cos\pi (2 r\varphi - F_{n}\varphi)
		- \cos \pi (F_{n}\varphi + 2\varepsilon)\big),
\end{align*}
where we have used the identity $\sin x\sin y = (\cos(x-y)-\cos(x+y))/2$ for the final step. Recall from \eqref{eq:GoldenRatioFibonacci} that $F_{n}\varphi = F_{n-1} -(-\varphi)^{n}$ for all $n \in \NN$. Thus, we get
\begin{align*}
	(P_{F_{n}}(\varphi,\varepsilon))^2\!
	=&(2 \sin \pi ( F_{n} \varphi + \varepsilon))^2 \\
	& \!\!\!\times \!\! \prod_{r=1}^{F_{n}-1} \!\!2 \left( \cos\pi (2 r\varphi - F_{n-1} + (-\varphi
		)^{n})	- \cos \pi (F_{n-1} - (-\varphi)^{n} \hspace{-0.05cm} + \hspace{-0.05cm} 2\varepsilon)
		\right)\\
	=&(2 \sin \pi ( F_{n} \varphi + \varepsilon))^2 (-1)^{(F_{n-1}+1)(F_{n}-1)} \\
	&\!\!\!\times \!\! \prod_{r=1}^{F_{n}-1}\!\! 2 \left( -\cos \pi (2 r\varphi + (-\varphi)^{n}) 
	+ \cos \pi ((-\varphi)^{n} - 2\varepsilon)\right).
\end{align*}
Note that $\mathrm{gcd}(F_{n-1},F_{n})=1$, and this implies $(-1)^{(F_{n-1}+1)(F_n-1)}=1$. We now use the identity $\cos x=1-2\sin^2 x/2$ to obtain
\begin{align*}
	(P_{F_{n}}(\varphi,\varepsilon))^2
	=&(2 \sin \pi( F_{n} \varphi + \varepsilon))^2 \\
	&\times \prod_{r=1}^{F_{n}-1}
	4 \left(\sin^2 \pi \left(r \varphi + \frac{(-\varphi)^{n}}{2}\right)- 
	\sin^2 \pi \left(\frac{(-\varphi)^{n}}{2} - \varepsilon\right) \right).
\end{align*}

Applying again the identity \eqref{eq:GoldenRatioFibonacci} as well as the substitution $t=F_{n-1}r \mod{F_{n}}$ we
follow \cite[Section 5]{MV16} to get
\begin{equation*}
	2\sin \pi (r \varphi + (-\varphi)^{n}/2) = s_{nt},
\end{equation*}
with $s_{nt}$ defined in \eqref{def:sn}.
Note that if $r$ runs through $\{1,\ldots,F_{n}-1\}$ then so does $t=F_{n-1}r \mod{F_{n}}$.
With $v_n$ defined as in \eqref{def:vn},
we finally have
\begin{align*}
	(P_{F_{n}}(\varphi,\varepsilon))^2
	=&(2 \sin \pi ( F_{n} \varphi + \varepsilon))^2\!
	\prod_{t=1}^{F_{n}-1}\!\left(s_{nt}^2-v_n^2\right)\\
	=&(2 \sin \pi ( F_{n} \varphi + \varepsilon))^2\!
	\prod_{t=1}^{F_{n}-1}\! s_{nt}^2\!
	\prod_{t=1}^{F_{n}-1}\!\left(1-\frac{v_n^2}{s_{nt}^2}\right)\\
	=&(2 \sin \pi ( F_{n} \varphi + \varepsilon))^2\!
	\prod_{t=1}^{F_{n}-1}\! s_{nt}^2 \!
	\prod_{t=1}^{F_{n}-1}\!\left(1-\frac{v_n^2}{s_{nt}^2}\right)\!
	F_{n}^2 \! \left(\prod_{t=1}^{F_{n}-1} 2\sin \pi \frac{t}{F_{n}} \right)^{-2} \\
	=&\left(\overline{A}_n(\varepsilon) B_{n} \overline{C}_n(\varepsilon)\right)^2,
\end{align*}
where we have used the well known product formula 
\begin{equation*}
	\prod_{r=1}^{q-1}2\sin\left(\frac{\pi r p}{q}\right)=q,
\end{equation*}
for positive integers $p,q \geq 1$ with $\mathrm{gcd}(p,q)=1$. 
\end{proof}


\section{Proofs}\label{sec:ProofMainThm}
Let us now turn to Lemma \ref{lem:boundslarge}. Fix some $N \in \N$, and let 
\begin{equation*}
N= \sum_{j=1}^{m} F_{n_j}
\end{equation*}
be its unique Zeckendorf representation. 
The product $P_N(\varphi)$ may be decomposed as 
\begin{equation*}
P_N(\varphi) = \prod_{j=1}^m \prod_{r=1}^{F_{n_j}} \left| 2 \sin \pi (r\varphi+ k_j \varphi) \right|,
\end{equation*}
where $k_j = \sum_{s=j+1}^m F_{n_s}$ for $1 \leq j \leq m-1$ and $k_m=0$. 
Using the notation introduced in \eqref{eq:shiftedpn}, we get
\begin{equation*}
P_N(\varphi) = \prod_{j=1}^m P_{F_{n_j}}(\varphi, k_j \varphi).
\end{equation*}
By applying again the identity $F_{n}\varphi= F_{n-1}-(-\varphi)^n$ from \eqref{eq:GoldenRatioFibonacci}, we have
\begin{equation*}
k_j \varphi = \sum_{s=j+1}^m \left( F_{n_s-1}-(-\varphi)^{n_s}\right),
\end{equation*}
and thus
\begin{equation}\label{eq:defe}
P_N(\varphi) = \prod_{j=1}^m P_{F_{n_j}}(\varphi, \varepsilon_j ) , \quad \varepsilon_j = - \sum_{s=j+1}^m (-\varphi)^{n_s}.
\end{equation}

We now observe that 
\begin{equation}
\label{eq:ratio}
P_{F_{n_j}}(\varphi, \varepsilon_j) = P_{F_{n_j}}(\varphi) \cdot \frac{P_{F_{n_j}}(\varphi, \varepsilon_j)}{P_{F_{n_j}}(\varphi)} = P_{F_{n_j}}(\varphi) \cdot \frac{\overline{A}_{n_j}(\varepsilon_j)\overline{C}_{n_j}(\varepsilon_j)}{A_{n_j}C_{n_j}}, 
\end{equation}
where $A_{n_j}$, $C_{n_j}$ and $\overline{A}_{n_j}(\varepsilon_j)$, $\overline{C}_{n_j}(\varepsilon_j)$ are defined in \eqref{Def:An}, \eqref{Def:Cn} and \eqref{eq:DefAiEpsi}, \eqref{eq:DefCiEpsi}, respectively. The claim in Lemma \ref{lem:boundslarge} is that $P_{F_{n_j}}(\varphi, \varepsilon_j) \geq 1$ whenever $j$ exceeds some threshold value (independent of $N$). Since we know from Theorem \ref{thm:MainResultMV} that $P_{F_{n_j}}(\varphi) \to 2.407\ldots > 12/5$ as $n_j \to \infty$, this will indeed follow from \eqref{eq:ratio} if we can show that
\begin{equation}\label{eq:AjCj}
\frac{\overline{A}_{n_j}(\varepsilon_j)\overline{C}_{n_j}(\varepsilon_j)}{A_{n_j}C_{n_j}} \geq \frac{5}{12}
\end{equation}
for all sufficiently large $n_j$. 

\subsection{The ratio $\overline{A}_{n_j}(\varepsilon_j)/A_{n_j}$}
We verify \eqref{eq:AjCj} by treating the two ratios $\overline{A}_{n_j}(\varepsilon_j)/A_{n_j}$ and $\overline{C}_{n_j}(\varepsilon_j)/C_{n_j}$ separately, starting with the simpler of the two.
\begin{lem}\label{lem:lowerboundA}
	Let $A_{n_j}$ and $\overline{A}_{n_j}(\varepsilon_j)$ be given in \eqref{Def:An} and \eqref{eq:DefAiEpsi}, respectively. We have 
		\begin{equation}
			\left|\frac{\overline{A}_{n_j}(\varepsilon_j)}{A_{n_j}} \right|=1+p_j +\mathcal{O}(\varphi^{2n_j}), 
		\end{equation}
	where the implied constant is independent of $n_j$,
		\begin{equation}
		\label{eq:defp}	
			p_j:= -\varepsilon_j (-\varphi)^{-n_j} =\sum_{s=j+1}^m(-\varphi)^{n_s-n_j} ,
		\end{equation}
	and $p_j \in [-\varphi^2, \varphi]$.
\end{lem}
\begin{proof}
Using the definition of $A_{n_j}$ and $\overline{A}_{n_j}(\varepsilon_j)$, we get 
\begin{equation*}
	\left|\frac{\overline{A}_{n_j}(\varepsilon_j)}{A_{n_j}} \right| = \left| \frac{\sin \pi \left( (-\varphi)^{n_j}-\varepsilon_j \right)}{\sin \pi \varphi^{n_j}} \right| = \left| \frac{\sin \pi (-\varphi)^{n_j}(1+p_j)}{\sin \pi \varphi^{n_j}}\right|,
\end{equation*}
with $p_j$ defined as in \eqref{eq:defp}. Since $n_1\geq 2$ and any two consecutive elements $n_s$ and $n_{s+1}$ must necessarily satisfy $n_{s+1}-n_{s} \geq 2$ (recall Definition \ref{def:zeckendorf}), it is easily seen that $p_j \in [-\varphi^2, \varphi]$. We finally apply $\sin x = x(1+\mathcal{O}(x^2))$ to obtain
\begin{align*}
	\left|\frac{\overline{A}_{n_j}(\varepsilon_j)}{A_{n_j}} \right| &= \frac{\pi \varphi^{n_j}(1+p_j) \left( 1+\mathcal{O}(\varphi^{2n_j}) \right)}{\pi \varphi^{n_j}\left( 1+\mathcal{O}(\varphi^{2n_j}) \right)} \\
	&= 1+p_j + \mathcal{O}(\varphi^{2n_j}) .
\end{align*}
\end{proof}

\subsection{The ratio $\overline{C}_{n_j}(\varepsilon_j)/C_{n_j}$}
We now shift our attention to the ratio $\overline{C}_{n_j}(\varepsilon_j)/C_{n_j}$. Our goal is to prove: 
\begin{lem}\label{lem:lowerboundC}
Let $C_{n_j}$ and $\overline{C}_{n_j}(\varepsilon_j)$ be given in \eqref{Def:Cn} and \eqref{eq:DefCiEpsi}, respectively. We have 
	\begin{equation}
		\frac{\overline{C}_{n_j}(\varepsilon_j)}{C_{n_j}} \geq 1-\frac{1}{7}(1+2p_j)^2 - \mathcal{O}(\varphi^{n_j/5}),
	\end{equation}
with $p_j$ as in \eqref{eq:defp} and where the implied constant is independent of $n_j$.
\end{lem}
The proof of Lemma \ref{lem:lowerboundC} is more elaborate than that of Lemma \ref{lem:lowerboundA}, and we start by stating two preliminary results. 
\begin{lem}[{\cite[Lemma 4.3]{MV16}}]\label{lem:prodbound}
	For $n\geq 2$ and real numbers $a_t$, $t =1,2, \ldots , n$, satisfying $A:=\sum_{t=1}^{n}|a_t|<1$, we have
	$$1-A < \prod_{t=1}^n (1-|a_t|) < \frac{1}{1-A}.$$
\end{lem}
Lemma \ref{lem:prodbound} is used in \cite{MV16} to show that the product $C_n$ in \eqref{Def:Cn} can be expressed as 
\begin{equation*}
C_n = \prod_{t=1}^{\infty} \left( 1-\frac{1}{u_t^2}\right) - \mathcal{O}(\varphi^{n/5}),
\end{equation*}
where
\begin{equation}\label{eq:defut}
u_t:=2\left(\sqrt{5}t -\{t \varphi\} +\frac{1}{2}\right) .
\end{equation}
We use it here to verify that a similar expression can be given for the perturbed product $\overline{C}_n(\varepsilon)$ whenever the perturbation $\varepsilon$ is sufficiently small.
\begin{lem}\label{lem:sumforC}
Let $\overline{C}_n(\varepsilon)$ be given in \eqref{eq:DefCiEpsi}, and assume that $|\varepsilon| \leq \varphi^{n+1}$. Then 
\begin{equation*}
\overline{C}_n(\varepsilon) \geq \prod_{t=1}^{\infty} \left( 1- \frac{(1-2\varepsilon(-\varphi)^{-n})^2}{u_t^2}\right) - \mathcal{O}(\varphi^{n/5}),
\end{equation*}
with $u_t$ given in \eqref{eq:defut} and where the implied constant is independent of $n$.
\end{lem}
\begin{proof}
Recall that
\begin{equation}\label{eq:Cnrestated}
\overline{C}_n(\varepsilon)=\prod_{t=1}^{(F_{n}-1)/2}\left(1-\frac{v_{n}^2}{s_{nt}^2}\right),
\end{equation}
where $v_n$ is given in \eqref{def:vn} and $s_{nt}$ is given in \eqref{def:sn}.
The assumption on $\varepsilon$ implies that $|\varepsilon (-\varphi)^{-n}| \leq \varphi$ and
\begin{equation}\label{eq:boundpert}
|1-2\varepsilon(-\varphi)^{-n}| \leq 1+2\varphi = \sqrt{5}.
\end{equation}
It thus follows from $\sin x = x (1+ \mathcal{O}(x^2))$ that
\begin{align}
	|v_n|
	&= 2 \left| \sin \frac{\pi}{2} (-\varphi)^{n} (1 - 2\varepsilon(-\varphi)^{-n})\right| \notag \\
	&=  \pi \varphi^{n} \, |1 - 2\varepsilon(-\varphi)^{-n}| \, (1+\mathcal{O}(\varphi^{2n})) \label{eq:viOrder} \\
	&\leq \pi \varphi^n\sqrt{5} \left( 1+\mathcal{O}(\varphi^{2n})\right) .\notag
\end{align}

We now split the product \eqref{eq:Cnrestated} at $\eta:=\left\lceil \varphi^{-3n/5}\right\rceil$, and treat first the terms where $t\geq \eta$. Using the bound \eqref{eq:viOrder} established for $|v_n|$, one can show that 
\begin{equation}\label{eq:rGreaterEta}
	\prod_{t=\eta+1}^{(F_{n}-1)/2}\left(1-\frac{v_n^2}{s_{nt}^2}\right)
	 = 1 - \mathcal{O}(\varphi^{n/5}).
\end{equation}
The argument is nearly identical to that given in \cite[p.~211]{MV16} for the unperturbed product $C_{n}$, and we therefore omit the details.

Now consider the terms in \eqref{eq:Cnrestated} where $t < \eta$. As deduced in \cite[p.~211]{MV16}, we have 
\begin{equation}	
	s_{nt} = \pi \varphi^{n}(u_t+\mathcal{O}(\varphi^{n/5})),
\end{equation}
with $u_t$ given in \eqref{eq:defut}, and combined with \eqref{eq:viOrder} this implies 
\begin{align*}
	\left|\frac{v_n}{s_{nt}}\right|
	&= \frac{\pi\varphi^{n} \, |1-2\varepsilon (-\varphi)^{-n}| \, (1+\mathcal{O}(\varphi^{2n}
	))}{\pi\varphi^{n}(u_t+\mathcal{O}(\varphi^{n/5}))}\\
	&= \frac{|1-2\varepsilon(-\varphi)^{-n}|}{u_t} \; \frac{1+\mathcal{O}(\varphi^{2n})}{1+\mathcal{O}(\varphi^{n/5})} \\
	&= \frac{|1-2\varepsilon (-\varphi)^{-n}|}{u_t}(1+\mathcal{O}(\varphi^{n/5})).
\end{align*}
It follows that
\begin{align}
	\prod_{t=1}^{\eta}\left(1-\frac{v_n^2}{s_{nt}^2}\right) 
	=& \prod_{t=1}^{\eta}\left(1-\frac{(1-2\varepsilon (-\varphi)^{-n})^2}{u_t^2} - \frac{\mathcal{O}(\varphi^{n/5})}{u_t^2}\right)\nonumber\\
	=&  \prod_{t=1}^{\eta}\left(1-\frac{(1-2\varepsilon (-\varphi)^{-n})^2}{u_t^2}\right) \label{eq:OneProduct}\\
	&\times \prod_{t=1}^{\eta}\left(1 - \frac{\mathcal{O}(\varphi^{n/5})}{u_t^2-(1-2\varepsilon (-\varphi)^{-n})^2}\right). \label{eq:TwoProducts}
\end{align}
We now evaluate the two subproducts \eqref{eq:OneProduct} and \eqref{eq:TwoProducts} separately. Starting with the former, we observe that since $\sum_{t=1}^{\infty}u_t^{-2} < 0.138$ (see \cite[p.~212]{MV16}), it follows from \eqref{eq:boundpert} that
\begin{equation*}
	(1-2\varepsilon (-\varphi)^{-n})^2\sum_{t=\eta +1}^{\infty} \frac{1}{u_t^2}<1.
\end{equation*}
By Lemma \ref{lem:prodbound} we then have
\begin{align*}
	0 &\leq \prod_{t=\eta+1}^{\infty}\left(1-\frac{(1-2\varepsilon (-\varphi)^{-n})^2}{u_t^2}\right) \\
	&\leq \left(1-(1-2\varepsilon (-\varphi)^{-n})^2 \sum_{t=1}^{\infty}\frac{1}{u_{t+\eta}^2}\right)^{-1} \\
	&\leq \left(1-\frac{1}{4}\sum_{t=1}^{\infty}\frac{1}{(t+\eta-1)^2} \right)^{-1} = \left( 1-\mathcal{O}(\eta^{-1}) \right)^{-1},\\
\end{align*}
and thus
\begin{align}\label{eq:FirstProductOrder}
	\prod_{t=1}^{\eta}\left(1-\frac{(1-2\varepsilon (-\varphi)^{-n})^2}{u_t^2}\right)
	&\geq \prod_{t=1}^{\infty}\left(1-\frac{(1-2\varepsilon (-\varphi)^{-n})^2}{u_t^2}\right) \big(1-\mathcal{O}(\eta^{-1})\big).
\end{align}

Now consider the second subproduct \eqref{eq:TwoProducts}. Using the bound \eqref{eq:boundpert}, it is easily checked that 
\begin{equation*}
\sum_{t=1}^{\infty}\frac{1}{u_t^2-(1-2\varepsilon (-\varphi)^{-n})^2}<\infty .
\end{equation*}
Thus, for sufficiently large $n$, we can use Lemma \ref{lem:prodbound} to conclude that 
\begin{align}
	\prod_{t=1}^{\eta}\left(1-\frac{\mathcal{O}(\varphi^{n/5})}{u_t^2-(1-2\varepsilon (-\varphi)^{-n})^2}\right) 
	&> 1- \mathcal{O}(\varphi^{n/5})\sum_{t=1}	^{\infty}\frac{1}{u_t^2-(1-2\varepsilon (-\varphi)^{-n})^2} \notag\\
	&= 1- \mathcal{O}(\varphi^{n/5}). \label{eq:SecondProductOrder}
\end{align}

Inserting the bounds \eqref{eq:FirstProductOrder} and \eqref{eq:SecondProductOrder} for the subproducts \eqref{eq:OneProduct} and \eqref{eq:TwoProducts}, respectively, we get 
\begin{equation}\label{eq:SmallerEta}
\prod_{t=1}^{\eta} \left( 1- \frac{v_n^2}{s_{nt}^2} \right) \geq \prod_{t=1}^{\infty}\left(1-\frac{(1-2\varepsilon (-\varphi)^{-n})^2}{u_t^2}\right) \big(1-\mathcal{O}(\eta^{-1})\big)
	\big(1 - \mathcal{O}(\varphi^{n/5})\big).
\end{equation}
Finally, inserting \eqref{eq:rGreaterEta} and \eqref{eq:SmallerEta} in \eqref{eq:Cnrestated}, and recalling that $\eta = \lceil \varphi^{-3n/5} \rceil$, we get
\begin{align*}
	\overline{C}_n(\varepsilon)
	&= \prod_{t=1}^{\eta}\left(1-\frac{v_n^2}{s_{nt^2}}\right)
	\prod_{t=\eta+1}^{(F_{n}-1)/2}\left(1-\frac{v_n^2}{s_{nt^2}}\right)\\
	&\geq \prod_{t=1}^{\infty}\left(1-\frac{(1-2\varepsilon (-\varphi)^{-n})^2}{u_t^2}\right) - \mathcal{O}(\varphi^{n/5}).
\end{align*}
\end{proof}

We are now equipped to bound the ratio $\overline{C}_{n_j}(\varepsilon_j)/C_{n_j}$ from below.

\begin{proof}[Proof of Lemma \ref{lem:lowerboundC}]
For $n=n_j$ and $\varepsilon=\varepsilon_j =  - \sum_{s=j+1}^m (-\varphi)^{n_s}$, we have $|\varepsilon|\leq \varphi^{n_j+1}$, and thus by Lemma \ref{lem:sumforC} we get
\begin{equation*}
\overline{C}_{n_j}(\varepsilon_j) \geq \prod_{t=1}^{\infty} \left( 1- \frac{(1+2p_j)^2}{u_t^2}\right) - \mathcal{O}(\varphi^{n_j/5}) ,
\end{equation*}
with $p_j$ given in \eqref{eq:defp}. From the definition it is clear that $C_{n_j} \leq 1$, and thus 
\begin{equation*}
\frac{\overline{C}_{n_j}(\varepsilon_j)}{C_{n_j}} \geq \overline{C}_{n_j}(\varepsilon_j) \geq \prod_{t=1}^{\infty} \left( 1- \frac{(1+2p_j)^2}{u_t^2}\right) - \mathcal{O}(\varphi^{n_j/5}).
\end{equation*}
Recall that $p_j \in [-\varphi^2, \varphi]$, and thus $(1+2p_j)^2\leq5$. Moreover, we recall from \cite[p.~212]{MV16} that $\sum_{t=1}^{\infty}1/u_t^2 < 1/7$. Thus, we may apply Lemma \ref{lem:prodbound} to obtain
\begin{align*}
\frac{\overline{C}_{n_j}(\varepsilon_j)}{C_{n_j}} &\geq 1- \sum_{t=1}^{\infty} \frac{(1+2p_j)^2}{u_t^2} - \mathcal{O}(\varphi^{n_j/5}) \\
&> 1- \frac{1}{7}(1+2p_j)^2 - \mathcal{O}(\varphi^{n_j/5}).
\end{align*}
\end{proof}

\subsection{Main proof}
Let us now confirm that Lemma \ref{lem:boundslarge} indeed follows from Lemmas \ref{lem:lowerboundA} and \ref{lem:lowerboundC}.
\begin{proof}[Proof of Lemma \ref{lem:boundslarge}]
We recall that our goal is to show that  
\begin{equation}\label{eq:greater1}
P_{F_{n_j}}(\varphi, \varepsilon_j) \geq 1
\end{equation}
whenever $n_j$ is sufficiently large. We have seen that
\begin{equation}
\label{eq:ratio2}
P_{F_{n_j}}(\varphi, \varepsilon_j) = P_{F_{n_j}}(\varphi) \cdot \frac{\overline{A}_{n_j}(\varepsilon_j)\overline{C}_{n_j}(\varepsilon_j)}{A_{n_j}C_{n_j}}, 
\end{equation}
where $A_{n_j}$, $C_{n_j}$ and $\overline{A}_{n_j}(\varepsilon_j)$, $\overline{C}_{n_j}(\varepsilon_j)$ are defined in \eqref{Def:An}, \eqref{Def:Cn} and \eqref{eq:DefAiEpsi}, \eqref{eq:DefCiEpsi}, respectively. 
By Lemmas \ref{lem:lowerboundA} and \ref{lem:lowerboundC} we get
\begin{equation*}
\frac{\overline{A}_{n_j}(\varepsilon_j)\overline{C}_{n_j}(\varepsilon_j)}{A_{n_j}C_{n_j}} \geq (1+p_j)\left( 1-\frac{1}{7}(1+2p_j)^2\right) + \mathcal{O}(\varphi^{n_j/5}),
\end{equation*}
with $p_j$ given in \eqref{eq:defp}. Consider the function 
\begin{equation*}
g(x) = (1+x)\left(1-\frac{1}{7}(1+2x)^2\right).
\end{equation*}
It is easily checked that for $x \in [-\varphi^2, \varphi]$, this function satisfies $g(x) > 5/11$, 
so there exists $S_1\in \N$ such that
\begin{equation*}
\frac{\overline{A}_{n_j}(\varepsilon_j)\overline{C}_{n_j}(\varepsilon_j)}{A_{n_j}C_{n_j}} > \frac{5}{12}
\end{equation*}
whenever $n_j \geq S_1$.
Moreover, since $P_{F_{n_j}}(\varphi)$ converges to a constant greater than $12/5$, there exists $S_2 \in \N$ such that $P_{F_{n_j}}(\varphi)\geq 12/5$ whenever $n_j \geq S_2$. Inserting these two inequalities in \eqref{eq:ratio2}, we obtain \eqref{eq:greater1} for all $n_j \geq \max \{S_1, S_2\}$. In particular, this means that \eqref{eq:greater1} holds for all $j \geq J= \max\{S_1, S_2\}/2$.
\end{proof}
Finally, we recall that Theorem \ref{thm:CountExLub} is a consequence of Lemma \ref{lem:boundslarge}:
\begin{proof}[Proof of Theorem \ref{thm:CountExLub}]
Let $N$ be any natural number, and let $N= \sum_{j=1}^{m} F_{n_j}$ be its unique Zeckendorf representation. We rewrite $P_N(\varphi)$ as 
\begin{equation*}
P_N(\varphi) = \prod_{r=1}^N |2\sin \pi r \varphi| = \prod_{j=1}^m P_{F_{n_j}}(\varphi, \varepsilon_j), 
\end{equation*}
with $\varepsilon_j$ given in \eqref{eq:defe}. 

Assume first that the length of the Zeckendorf representation of $N$ is smaller than the bound $J$ in Lemma \ref{lem:boundslarge}, i.e.\ $m\leq J$. In this case it follows from Lemma \ref{lem:boundshift} that 
\begin{equation}\label{eq:bound1}
P_N(\varphi) \geq K_1^m \geq K_1^J.
\end{equation}
for some $0<K_1\leq 1$. 

Suppose now that $m>J$. Then by Lemmas \ref{lem:boundshift} and \ref{lem:boundslarge}, we have
\begin{equation}\label{eq:bound2}
P_N(\varphi) = \left( \prod_{j=1}^J P_{F_{n_j}}(\varphi, \varepsilon_j) \right) \left(\prod_{j=J+1}^m P_{F_{n_j}}(\varphi, \varepsilon_j)\right) \geq K_1^J \cdot 1^{m-J} \geq K_1^J.
\end{equation}

Combining \eqref{eq:bound1} and \eqref{eq:bound2} we have $P_N(\varphi)\geq K_1^J$ for all $N$, where $J \in \N$ and $K_1>0$ are absolute constants. It follows that 
\begin{equation*}
\liminf_{N\to \infty} P_N(\varphi) \geq K_1^J > 0.
\end{equation*}
\end{proof}


\section{Concluding remarks}

\subsection{Theorem \ref{thm:CountExLub} for quadratic irrationals}\label{sec:quadirrat}
It is natural to ask whether 
\begin{equation}\label{eq:extquadirr}
\liminf_{N \to \infty} P_N(\alpha)>0
\end{equation}
for any quadratic irrational $\alpha$. Such an extension of Theorem \ref{thm:CountExLub} seems plausible in light of \cite{GrNe18}, where an extension of Theorem~\ref{thm:MainResultMV} to all quadratic irrationals is given (see \cite[Theorem 1.2]{GrNe18}). The authors feel confident that using this extension and the strategy lined out in Section~\ref{sec:strategy}, it should be possible to verify \eqref{eq:extquadirr} for other quadratic irrationals $\alpha$. Nevertheless, there are certain technical challenges involved in switching from the Zeckendorf representation of the integer $N$ to the more general Ostrowski representation of $N$ (as one will then have to do). We leave this question open for the curious reader. For information on Ostrowski representations see for example \cite{AS03} or \cite{RS92}.

\subsection{$P_1(\varphi)$ as a lower bound for $P_N(\varphi)$}
Numerical calculations seem to suggest that for $\varphi=(\sqrt{5}-1)/2$, we actually have
\begin{equation}\label{eq:p1ineq}
P_N(\varphi) \geq P_1(\varphi) = 1.86\ldots ,
\end{equation}
for all $N \in \N$. This inequality would provide a significantly greater lower bound for $\liminf_{N\to \infty} P_N(\varphi)$ than what is attained in our current proof of Theorem \ref{thm:CountExLub}.

More generally, numerical experiments suggest that for $n\geq 3$ and $N \in \{F_{n-1},\ldots,F_n-1\}$ we have
\begin{equation}\label{con:LocBehavPN}
	 P_{F_{n-1}}(\varphi) \leq P_{N}(\varphi) \leq P_{F_n-1}(\varphi).
\end{equation}
Confirming the left inequality in \eqref{con:LocBehavPN} would provide further support for \eqref{eq:p1ineq}, as $P_{F_n}(\varphi)$ appears to converge rapidly towards $2.407\ldots$. On the other hand, should the right inequality in \eqref{con:LocBehavPN} be true, then
\begin{align*}
	P_N(\varphi)\leq P_{F_{n}-1}(\varphi) \leq cF_{n} \leq2cF_{n-1} \leq 2cN,
\end{align*} 
where we have used that $P_{F_n-1}(\varphi) \leq c F_{n}$ (see \cite{MV16}). This would indicate that $P_N(\varphi)$ grows at most linearly in $N$, improving significantly on all known bounds for the asymptotic growth of $P_N(\varphi)$.

\subsection*{Addendum}
It was brought to our attention in April 2023 that the main result in this paper (Theorem \ref{thm:CountExLub}) was in fact established by Verschueren in his PhD thesis in 2016 \cite{VerPhD16}. We regret not being aware of this work, and urge others to cite it when referring to Theorem \ref{thm:CountExLub} in the future.

\Addresses
	
\end{document}